\documentclass[12pt]{amsart}

\usepackage{amsmath}
\usepackage{amsthm}
\usepackage{amssymb}
\usepackage[margin=1in]{geometry}
\usepackage{tikz}
\usetikzlibrary{calc}
\usepackage{setspace}
\usepackage{microtype}
\usepackage{enumitem}
\usepackage[all]{xy}

\usepackage{nicefrac}

\newtheorem{theorem}{Theorem}
\newtheorem{corollary}[theorem]{Corollary}
\newtheorem{lemma}[theorem]{Lemma}

\theoremstyle{definition}

\newtheorem{remark}{Remark}

\newtheorem*{example*}{Example}
\newtheorem*{examples*}{Examples}

\newcommand{\abs}[1]{\left\lvert #1 \right\rvert}

\newcommand{\set}[1]{\left\{ #1 \right\}}
\renewcommand{\P}{\mathbb P}
\renewcommand{\Im}{\operatorname{Im}}

\newcommand{\R}{\mathbb R}
\newcommand{\C}{\mathbb C}
\newcommand{\Z}{\mathbb Z}
\newcommand{\rat}{\dashrightarrow}
\newcommand{\Kbar}{\bar{K}}

\newcommand{\dyadic}{\mathbb Z\!\left[ \frac{1}{3} \right]}
\newcommand{\parauto}{G^+_p}
\newcommand{\bigauto}{G^\pm_p}
\newcommand{\bigautoev}{G^{\pm,\textrm{ev}}_p}

\DeclareMathOperator{\id}{id}
\DeclareMathOperator{\Stab}{Stab}
\DeclareMathOperator{\NS}{NS}
\DeclareMathOperator{\Gal}{Gal}
\DeclareMathOperator{\Pic}{Pic}
\DeclareMathOperator{\Aut}{Aut}

\DeclareMathOperator{\Diff}{Diff}
\newcommand{\Bl}{\mathrm{Bl}}
\DeclareMathOperator{\Hom}{Hom}

\DeclareMathOperator{\PGL}{PGL}
\DeclareMathOperator{\GL}{GL}

\title[A non-finitely generated automorphism group]{A projective variety with discrete, non-finitely generated automorphism group}
\author{John Lesieutre}
\address{Department of Mathematics, Statistics and Computer Science \\ University of Illinois at Chicago, Chicago, IL 60607}
\email{jdl@uic.edu}

\begin{document}

\begin{abstract}
We construct a projective variety with discrete, non-finitely generated automorphism group.  As an application, we show that there exists a complex projective variety with infinitely many non-isomorphic real forms.
\end{abstract}

\maketitle

\section{Introduction}

Suppose that \(X\) is a projective variety over a field \(K\).  The set of automorphisms of \(X\) can be given the structure of a \(K\)-scheme by realizing it as an open subset of \(\Hom(X,X)\).  In general, \(\Aut(X)\) is locally of finite type, but it may have countably many components, arising from components of the Hilbert scheme.  Write \(\pi_0(\Aut(X)) = (\Aut(X)/\Aut^0(X))_{\Kbar}\) for the group of geometric components.

\begin{examples*}  \leavevmode
\begin{enumerate}
\item Let \(X = \P^r\).  Then \(\Aut(X) \cong \Aut^0(X) \cong \PGL_{r+1}(K)\), and \(\pi_0(\Aut(X))\) is trivial.
\item Let \(E\) be a general elliptic curve over \(K\).  Then \(\pi_0(\Aut(E \times E)) \cong \GL_2(\Z)\) is an infinite discrete group.
\item Let \(X\) be a general hypersurface of type \((2,2,2)\) in \(\P^1 \times \P^1 \times \P^1\).  Then \(X\) is a K3 surface, and the covering involutions associated to the three projections \(X \to \P^1 \times \P^1\) generate a subgroup of \(\pi_0(\Aut(X))\) isomorphic to \(\Z/2\Z \ast \Z/2\Z \ast \Z/2\Z\)~\cite{cantatsurfaces}.
\end{enumerate}
\end{examples*}

According to a result of Brion~\cite{brion}, any connected algebraic group over a field of characteristic \(0\) can be realized as \(\Aut^0(X)\) for some smooth, projective variety.  In contrast, very little seems to be known in general about the component group \(\pi_0(\Aut(X))\).  In what follows, let \(K\) be a field of characteristic \(0\), not necessarily algebraically closed, and let \(\Kbar\) be an algebraic closure.  All varieties are defined over \(K\), except where noted otherwise, and by a point we mean a \(K\)-point.  Our first result is the following.

\begin{theorem}
\label{autoexample}
There exists a smooth, geometrically simply connected, projective variety \(X\) over \(K\) for which \(\pi_0(\Aut(X))\) is not finitely generated.  
\end{theorem}

The question of finite generation of \(\pi_0(\Aut(X))\) has been raised several times in various arithmetic~\cite{mazur},\cite{mothread}  and geometric~\cite{brion},\cite{cantatkummer} contexts.

The automorphism group owes its arithmetic interest in part to its close relation with the forms of a variety over an extension field. If \(X\) is a variety over a field \(K\), and that \(L\) is a Galois extension of \(K\), then an \(L/K\)-\emph{form} of \(X\) is a variety \(X'\) over \(K\) for which \(X_L \cong X^\prime_L\).  The set of \(L/K\)-forms of \(X\) is in bijection with the Galois cohomology set \(H^1(\Gal(L/K),\Aut(X_L))\), and we will construct a variety with infinitely many \(L/K\)-forms by exhibiting a variety for which \(\Aut(X_L)\) is pathological.

\begin{theorem}
\label{twistexample}
Suppose that \(L/K\) is a quadratic extension.  Then there exists a \(K\)-variety \(X^\prime\) with infinitely many \(L/K\)-forms.
\end{theorem}

The component group \(\pi_0(\Aut(X))\) is an algebraic analog of the mapping class group \(\pi_0(\Diff(M))\) of a smooth manifold \(M\).  In general, the mapping class group is not finitely generated, with an example provided by tori in dimension at least five~\cite{hatcher}.  However, at least in high dimensions, the failure  of \(\pi_0(\Diff(M))\) to be finitely generated is attributable to the fundamental group of \(M\): according to a theorem of Sullivan~\cite{sullivan}, if \(\dim M \geq 5\) and \(\pi_1(M) = 0\), then \(\pi_0(\Diff(M))\) is finitely generated.  The example presented here is simply connected, and the failure of \(\pi_0(\Aut(X))\) to be finitely generated is not due to its topology.

Before giving the example, we sketch the technique.  If \(X\) is a variety and \(Z\) is a closed subscheme of \(X\), then the automorphisms of \(X\) that lift to automorphisms of the blow-up \(\Bl_Z(X)\) are precisely those that map \(Z\) to itself (not necessarily fixing \(Z\) pointwise).  Our approach, roughly speaking, is to find a variety \(X\) with a subscheme \(Z\) so that \(\Stab(Z) \subset \Aut(X)\) is not finitely generated, and then to pass to the blow-up \(\Bl_Z(X)\) to obtain a variety realizing \(\Stab(Z)\) as an automorphism group.
There are two main difficulties.  The first is to find \(X\) and \(Z\) for which the stabilizer of \(Z\) in \(\Aut(X)\) is not finitely generated.  The second is to ensure that \(\Bl_Z(X)\) does not have any automorphisms other than those lifted from \(X\).  

To prove that our variety \(X\) has non-finitely generated automorphism group, we will exhibit a smooth rational curve \(C\) which is fixed by every automorphism of \(X\).  Restriction of automorphisms then determines a map
\(\rho: \Aut(X) \to \Aut(C) \cong \PGL_2(K)\). We arrange that the image of \(\rho\) is contained in an abelian subgroup of \(\PGL_2(K)\) and exhibit an explicit non-finitely generated subgroup of \(\Im(\rho)\).  It follows that  \(\Aut(X)\) is not finitely generated.

We turn now to the construction.  Given a subvariety \(V \subset X\), write
\[
\Aut(X;V) = \set{\phi \in \Aut(X): \text{$\phi(V) = V$}}
\]
There are three main steps.  First, we describe a family of elliptic rational surfaces \(S\) for which \(\Aut(S)\) is a large discrete group, and there is a rational curve \(C\) on \(S\) with \(\Aut(S;C)\) of finite index.  Second, we specialize the surface \(S\) in order to control the image of \(\Aut(S;C) \to \Aut(C) \cong \PGL_2(K)\).  We arrange that there is a point \(p\) on \(C\) so that the subgroup \(\parauto\) of automorphisms \(\phi\) so that \(\phi\vert_C\) is parabolic with fixed point \(p\) is not finitely generated.  At last, by some auxiliary constructions, we arrive at a six-dimensional variety \(X\) whose automorphisms are precisely given by \(\parauto\).

\section*{Step 1: Automorphisms of surfaces with prescribed action on a curve}

If \(z_1\), \(z_2\), \(z_3\), and \(z_4\) are four distinct points in \(\P^1\), there is a unique involution \(\imath : \P^1 \to \P^1\) with \(\imath(z_1) = z_2\) and \(\imath(z_3) = z_4\), which is defined over \(K\).  Figure~\ref{conicinvo} shows how this map can be constructed geometrically when \(\P^1\) is embedded as a conic in \(\P^2\).

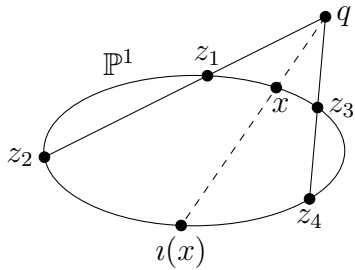
\begin{figure}[hbt]
\centering
\begin{tikzpicture}
    \draw (0,0) ellipse (2 and 1);

    \coordinate (P1) at ($(85:2 and 1)$);
    \filldraw (P1) circle (2pt) node[above] {$z_1$};
    \coordinate (P2) at ($(185:2 and 1)$);
    \filldraw (P2) circle (2pt) node[left] {$z_2$};
    \coordinate (P3) at ($(35:2 and 1)$);
    \filldraw (P3) circle (2pt) node[right] {$z_3$};
    \coordinate (P4) at ($(320:2 and 1)$);
    \filldraw (P4) circle (2pt) node[below] {$z_4$};

    \coordinate (X) at ($(265:2 and 1)$);
    \filldraw (X) circle (2pt) node[below] {$\imath(x)$};

    \coordinate (L) at ($(120:2 and 1)$);
    \filldraw (L) node[above] {$\P^1$};

    \coordinate (Q) at (intersection of P1--P2 and P3--P4);
    \filldraw (Q) circle (2pt) node[right] {$q$};

    \draw (Q)--(P2);
    \draw (Q)--(P4);

    \draw [dashed] (Q)--(X);
    \coordinate (X) at ($(57:2 and 1)$);
    \filldraw (X) circle (2pt) node[below] {$\;x$};

\end{tikzpicture}
\caption{Geometric construction of $\imath$}
\label{conicinvo}
\end{figure}

Given an ordered \(5\)-tuple \(P = (p_1,p_2,p_3,p_4,p_5)\) of points in \(\P^1\), let \(\Gamma_P \subset \PGL_2(K)\) be the subgroup generated by the involutions \(\imath_{ij,kl} : \P^1 \to \P^1\) satisfying \(\imath(p_i) = p_j\) and \(\imath(p_k) = p_l\), where \(i\), \(j\), \(k\) and \(l\) are distinct indices. 
For a given configuration \(P\), there are \(15\) such involutions for different choices of points.

\begin{theorem}
\label{surfaceautos}
Suppose that \(P\) is a configuration of five distinct points in \(\P^1\).   There exists a smooth rational surface \(S\) containing a rational curve \(C \subset S\) such that
\begin{enumerate}
\item \(\Aut(S)\) is discrete;
\item \(\Aut(S;C)\) has finite index in \(\Aut(S)\);
\item The image of \(\rho: \Aut(S;C) \to \Aut(C)\) contains \(\Gamma_P\).
\end{enumerate}
\end{theorem}
\begin{proof}

Let \(L_0,\ldots,L_5\) be six lines in \(\P^2\) intersecting at \(15\) distinct points \(p_{ij} = L_i \cap L_j\), and let \(S\) be the blow-up of \(\P^2\) at these \(15\) points, with exceptional divisor \(E_{ij}\) over \(p_{ij}\).
Write \(R\) for a partition of the six lines into two sets of three, with a distinguished line in each set. Given such a labelling, denote by \(L_{R,0}\), \(L_{R,1}\), \(L_{R,2}\) and  \(L_{R,0}^\prime\), \(L_{R,1}^\prime\), \(L_{R,2}^\prime\) the two triples, with \(L_{R,0}\) and \(L_{R,0}^\prime\) the two distinguished lines.  Let \(O_R\) be the point of intersection of \(L_{R,0}\) and \(L_{R,0}^\prime\).

The choice of a labelling \(R\) determines two completely reducible cubics \(\Gamma = L_{R,0} \cup L_{R,1} \cup L_{R,2}\) and \(\Gamma^\prime = L_{R,0}^\prime \cup L_{R,1}^\prime \cup L_{R,2}^\prime\), which span a pencil in \(\P^2\).  The base locus of the pencil is the nine points \(L_{R,i} \cap L_{R,j}^\prime\).  Let \(\pi_R : S_R \to \P^2\) be the blow-up of \(\P^2\) at these points, so that the pencil gives rise to an elliptic fibration \(\gamma_R : S_R \to \P^1\).   Note that the fibration \(\gamma_R\) must be relatively minimal (i.e.\ there are no \((-1)\)-curves contained in the fibers): a general fiber is linearly equivalent to \(-K_{S_R}\), and so a \((-1)\)-curve on \(S_R\) must have intersection \(1\) with every fiber.

The exceptional divisor of \(\pi_R\) above the point \(O_R\) provides a section \(E\) of \(\gamma_R\).  Let \(\imath_R : S_R \rat S_R\) be the birational involution induced by the \(\gamma_R\)-fiberwise action of \(x \mapsto -x\) on the smooth fibers, with the section \(E\) as the identity. 
Since \(\gamma_R\) is a relatively minimal fibration, \(\imath_R\) extends to a regular map \(\imath_R : S_R \to S_R\) on the entire surface \(S_R\) (see e.g.~\cite[II.10, Theorem 1]{shafarevich}).  Such a map necessarily permutes the three nodes on each of the  fibers \(\Gamma\) and \(\Gamma^\prime\), and so lifts to a biregular involution on the fifteen point blow-up \(S\).

In fact, \(\imath_R\) is a lift of the classical Bertini involution of \(\P^2\) centered at the eight points \(L_{R,i} \cap L_{R,j}\) with \((i,j) \neq (0,0)\).
There is a simple geometric description of \(\imath_R\) as a rational map of \(\P^2\), and in particular of its action on the curve \(L_{R,0}\).  Suppose that \(\ell\) is a line in \(\P^2\) passing through the point \(O_R\), and that \(x\) is a point on \(\ell\) lying on a smooth fiber \(C_x\) of \(\gamma_R\).  Then \(\ell\) meets \(C_x\) at \(x\), \(O_R\), and the third point \(\imath_R(x)\). This description remains valid on components of the singular fibers not containing \(\ell\), and so \(\imath_R\) acts on \(\ell\) so that the two points \(L_{R,1} \cap \ell\) and \(L_{R,2} \cap \ell\) are exchanged, as are \(L_{R,1}^\prime \cap \ell\) and \(L_{R,2}^\prime \cap \ell\). This uniquely determines the map: if \(\ell\) is any line through \(O_R\) for which the four points \(L_{R,1} \cap \ell\), \(L_{R,2} \cap \ell\), \(L_{R,1}^\prime \cap \ell\) and \(L_{R,2}^\prime \cap \ell\) are distinct, including \(L = L_{R,0}\), then \(\imath_R\) restricts to \(\ell\) as the unique involution exchanging these two pairs of points.

The rational surface \(S\) claimed by the theorem can now be constructed by choosing the lines in special position.  Fix a line \(C = L_0 \subset \mathbb P^2\), and choose five other lines \(L_1,\ldots,L_5\) so that \(L_i \cap C = p_i\), where the \(p_i\) are the points of the configuration \(P\).  Since the field \(K\) is infinite, for general choices of the \(L_i\), the fifteen points of intersection are distinct.  The involution of \(C\) exchanging \(p_i\) with \(p_j\) and \(p_k\) with \(p_l\) is realized as the restriction of \(\imath_R : S \to S\) for a suitable labelling \(R\): let \(m\) be the unique index which does not appear among \(i\), \(j\), \(k\), and \(l\), and take \(L_{R,0} = C\), \(L_{R,1} = L_i\), \(L_{R,2} = L_j\), \(L_{R,0}^\prime = L_m\), \(L_{R,1}^\prime = L_k\), and \(L_{R,2}^\prime = L_l\).  Thus each involution \(\imath_{ij,kl}\) on \(C\) is the restriction of an automorphism of \(\imath_R : S \to S\) fixing \(C\), as claimed.

A blow-up \(X\) of \(\P^2\) at four points with no three collinear satisfies \(H^0(X,TX) = 0\), and so \(\Aut^0(S)\) is trivial since \(S \to \P^2\) factors through such a blow-up.  It remains only to check that the subgroup \(\Aut(S;C)\) has finite index in \(\Aut(S)\).  This is a consequence of the fact that \(S\) is a Coble rational surface~\cite{coblesurface}, \cite{cantatdolgachev}: the linear system \(\abs{-2K_S}\) has a unique element, the union of the strict transforms of the six lines \(L_i\). Indeed, each line satisfies \(-2K_S \cdot L_i = -4\), and so must be contained in the base locus of \(\abs{-2K_S}\).
An automorphism preserves the anticanonical class, so the six lines are permuted by any element of \(\Aut(S)\), giving rise to a map \(\Aut(S) \to S_6\).  The subgroup \(\Aut(S;C)\) is the preimage of the subgroup of permutations fixing \(C\), and thus of finite index.
\end{proof}

\begin{figure}[hbt]
\centering
\begin{tikzpicture}
\node[anchor=south west,inner sep=0] at (0,0) {\includegraphics[scale=0.5]{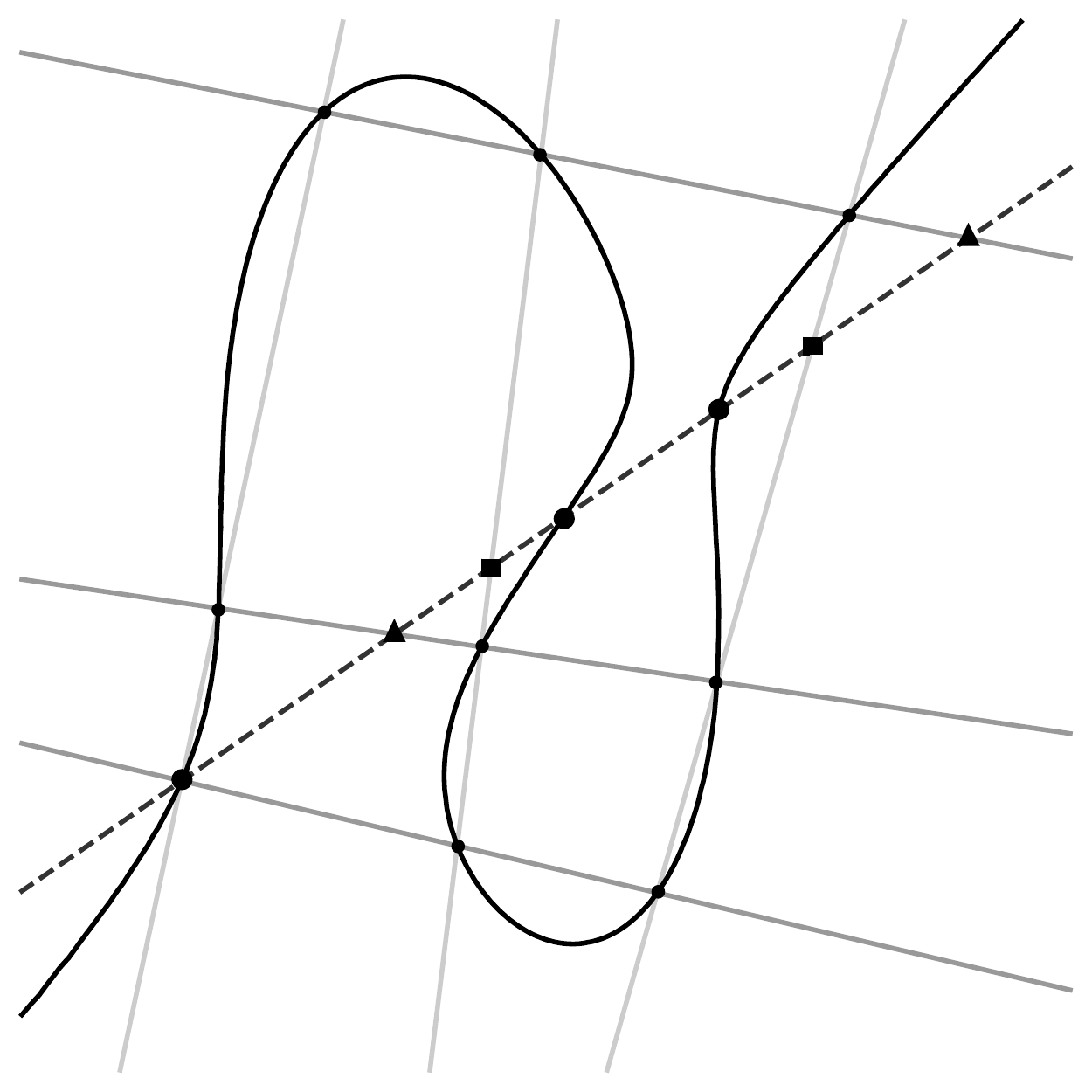}};
\node at (1.4,1.5) {$O_R$};
\node at (1.03,4.5) {$C_x$};
\node at (3.43,3.13) {$x$};
\node at (1.85,2.1) {$\ell$};
\node [fill=white,inner sep=1pt] at (4.8,3.7) {$\imath_R(x)$};
\end{tikzpicture}
\caption{Construction of \(\imath_R : S \to S\)}
\label{irconstruct}
\end{figure}

Figure~\ref{irconstruct} illustrates the geometry of the map \(\imath_R\).  The restriction of \(\imath_R\) to the line \(\ell\) is the unique involution exchanging the two points marked ``\(\blacktriangle\)'' and the two points marked ``\footnotesize\(\blacksquare\)\normalsize''.

\begin{remark}
\label{extragenerality}
Consider the four lines through the point \(p_{05}\) given by \(L_0\), \(L_5\), \(L(p_{05},p_{14})\), and \(L(p_{05},p_{23})\), which define four points in \(\P T_{p_{05}}(\P^2) \cong \P^1\).  It will later be convenient to assume that there is no automorphism of this \(\P^1\) which fixes the first two points while exchanging the third and fourth; this will be the case for general choices of the five lines even after the intersections with \(L_0\) are prescribed.
\end{remark}

\section*{Step 2: Specializing the configuration \(P\)}

We now exhibit a configuration \(P = (p_1,p_2,p_3,p_4,p_5)\) for which the group \(\Gamma_P\) contains two particular transformations with a common fixed point.  Fix coordinates on \(\P^1\).

\begin{lemma}
\label{specialconfig}
For the configuration
\[
P = (0,1,2,3,6)
\]
the group \(\Gamma_P\) contains the two elements
\[
\sigma = \begin{pmatrix}
1 & 1 \\ 0 & 1 
\end{pmatrix}, \quad
\tau = \begin{pmatrix}
3 & 0 \\ 0 & 1 
\end{pmatrix}.
\]
\end{lemma}
\begin{proof}
We claim that \(\sigma = \imath_{12,34} \circ \imath_{13,45} \circ \imath_{13,25}\) and \(\tau = \imath_{13,45} \circ \imath_{24,35} \circ \imath_{15,23}\).  Indeed,
\begin{align*}
&(\imath_{12,34} \circ \imath_{13,45} \circ \imath_{13,25})(p_1) = (\imath_{12,34} \circ \imath_{13,45})(p_3) = \imath_{12,34}(p_1) = p_2, \\
&(\imath_{12,34} \circ \imath_{13,45} \circ \imath_{13,25})(p_2) = (\imath_{12,34} \circ \imath_{13,45})(p_5) = \imath_{12,34}(p_4) = p_3, \\
&(\imath_{12,34} \circ \imath_{13,45} \circ \imath_{13,25})(p_3) = (\imath_{12,34} \circ \imath_{13,45})(p_1) = \imath_{12,34}(p_3) = p_4.
\end{align*}
For the configuration \(P\), this yields \((\imath_{12,34} \circ \imath_{13,45} \circ \imath_{13,25})(0) = 1\), \((\imath_{12,34} \circ \imath_{13,45} \circ \imath_{13,25})(1) = 2\), and \((\imath_{12,34} \circ \imath_{13,45} \circ \imath_{13,25})(2) = 3\), so the composition must be the automorphism \(\sigma\) given by \(z \mapsto z+1\).
Similarly, 
\begin{align*}
&(\imath_{13,45} \circ \imath_{24,35} \circ \imath_{15,23})(p_1) = (\imath_{13,45} \circ \imath_{24,35})(p_5) = \imath_{13,45}(p_3) = p_1, \\
&(\imath_{13,45} \circ \imath_{24,35} \circ \imath_{15,23})(p_2) = (\imath_{13,45} \circ \imath_{24,35})(p_3) = \imath_{13,45}(p_5) = p_4, \\
&(\imath_{13,45} \circ \imath_{24,35} \circ \imath_{15,23})(p_3) = (\imath_{13,45} \circ \imath_{24,35})(p_2) = \imath_{13,45}(p_4) = p_5.
\end{align*}

For the configuration \(P\), this map sends \(0\) to \(0\), \(1\) to \(3\), and \(2\) to \(6\), and so must be the automorphism \(\tau\) given by \(z \mapsto 3z\).
\end{proof}

Let \(S\) be the rational surface constructed in Theorem~\ref{surfaceautos} corresponding to the configuration \(P\) given by Lemma~\ref{specialconfig}, 
so that the image of \(\rho : \Aut(S;C) \to \Aut(C)\) contains the elements \(\sigma\) and \(\tau\).  
Write \(\parauto\) for the subgroup of \(\Aut(S;C)\) given by automorphisms which restrict to \(C\) as parabolic transformations fixing the point \(p\); if coordinates are chosen so that \(p = \infty\), these are the automorphisms 
restricting to translations \(z \mapsto z+c\).  Equivalently, we ask that \(\phi\vert_C\) is given by a unipotent upper triangular matrix \(\left( \begin{smallmatrix} 1 & c \\ 0 & 1 \end{smallmatrix} \right)\).

\begin{lemma}
\label{surfacenonfg}
Let \(p = \infty\) on the curve \(C\).  Then \(\parauto\) is not a finitely generated group.
\end{lemma}
\begin{proof}
An automorphism \(\phi\) in \(\Aut(S;C)\) lies in \(\parauto\) if and only if \(\rho(\phi)\) lies in the subgroup \(B \subset \PGL_2(K)\) comprising matrices of the form
\[
\begin{pmatrix} 1 & c \\ 0 & 1 \end{pmatrix},
\]
which correspond to parabolic M\"obius transformations \(z \mapsto z+c\).
The subgroup \(B\) is abelian, isomorphic to \(\mathbb G_a\); since \(\rho(\parauto)\) is contained in \(B\), this group is abelian as well.
For any integer \(n\), the transformation
\[
\tau^{-n} \circ \sigma \circ \tau^{n} = \begin{pmatrix} 1 & \nicefrac{1}{3^{n}} \\ 0 & 1 \end{pmatrix}
\]
is contained in \(B\).  Since \(\sigma\) and \(\tau\) both lie in \(\Im(\rho)\) by
the construction of Theorem~\ref{surfaceautos}, the elements \(\tau^{-n} \circ \sigma \circ \tau^n\) all lie in \(\rho(\parauto)\), and so \(\rho(\parauto)\) has a subgroup isomorphic to \(\dyadic\).  Since \(\rho(\parauto)\) is abelian and has a non-finitely generated subgroup, it is not finitely generated. A quotient of a finitely generated group is finitely generated, and we conclude that \(\parauto\) itself is not finitely generated.
\end{proof}

The following geometric characterization of elements of \(\parauto\) will prove useful.  Let \(\Delta_S : S \to S \times S\) be the diagonal map.

\begin{lemma}
\label{fixingp}
Suppose that \(\phi : S \to S\) is an automorphism fixing \(p\).  Then \(\phi\) fixes \(C\) as well.  Furthermore, \(\phi\) lies in \(\parauto\) if and only if \(\id_S \times \phi : S \times S \to S \times S\) fixes the tangent direction \(T_{\Delta_S(p)}(\Delta_S(C))\).
\end{lemma}
\begin{proof}
Any automorphism of \(S\) permutes the components of \(\abs{-2K_S}\), which are the strict transforms of the six lines \(L_i\).  The only component containing \(p\) is \(C\) itself, and so \(C\) must be invariant under \(\phi\).

An automorphism fixing \(C\) and \(p\) lies in \(\parauto\) if and only if \(p\) is a fixed point of \(\phi\vert_C\) with multiplicity \(2\), which is the case if and only if \(\id_S \times \phi : S \times S \to S \times S\) fixes \(\Delta_S(p)\) and the tangent direction \(T_{\Delta_S(p)}(\Delta_S(C))\), so that \((\id_S \times \phi)(\Delta_S(C))\) is tangent to the diagonal at \(\Delta_S(p)\).
\end{proof}

\begin{remark}
\label{notinjective}
Let \(\bar{\sigma}\) and \(\bar{\tau}\) be automorphisms of \(S\) which restrict to \(C\) as \(\sigma\) and \(\tau\), as constructed in Theorem~\ref{surfaceautos}.  Although the restrictions to \(C\) of the automorphisms \(\mu_m = \bar{\tau}^{-m} \circ \bar{\sigma} \circ \bar{\tau}^m\) and \(\mu_n = \bar{\tau}^{-n} \circ \bar{\sigma} \circ \bar{\tau}^n\) commute and satisfy \(\mu_{n-1}\vert_C \circ \mu_n\vert_C^{-3} = \id_C\), these maps do not commute as automorphisms of \(S\), and the map \(\Aut(S;C) \to \Aut(C)\) is not injective.  For example, the commutator \([\mu_0,\mu_1]\) is an automorphism of \(S\) which restricts to \(C\) as the identity, but a straightforward if somewhat tedious computation of the action of the involutions \(\imath_R\) on \(\NS(S)\) shows that the induced map \([\mu_0,\mu_1] : \P^2 \rat \P^2\) is a Cremona transformation of degree \(1,944,353\) with first dynamical degree \(\lambda_1 \approx 752,897\).
It seems conceivable that \(\parauto\) is a free group on the countably many generators \(\mu_n\), though this is difficult to prove.
\end{remark}

\begin{remark}
\label{badsurface}
The kernel \(G\) of \(\Aut(S;C) \to \Aut(C)\) is also of interest: this is the subgroup of automorphisms which fix \(C\) pointwise, including the maps \([\mu_m,\mu_n]\) of the remark above.   It seems likely that \(G\) is not finitely generated; if this is the case, then by choosing a very general point \(q\) on \(C\), we might obtain a rational surface \(S^\prime = \Bl_q S\) such that \(\Aut(S^\prime)\) is isomorphic to \(G\) and is not finitely generated.  However, it is not clear how to prove either that \(G\) is not finitely generated, or that the blow-up does not admit automorphisms other than those lifted from \(S\).
\end{remark}

\section*{Step 3: A variety with non-finitely generated $\Aut(X)$}

We now construct a higher-dimensional variety \(X\) realizing \(\parauto\) as \(\Aut(X)\).  Although \(\parauto\) is not the stabilizer of any closed subscheme of \(S\), it is the stabilizer of a closed subscheme of \(S \times S\) under the group of automorphisms of \(S \times S\) of the form \(\id_S \times \phi\): an automorphism \(\phi\) lies in \(\parauto\) and only if \(\id_S \times \phi\) fixes both \(\Delta_S(p)\) and the tangent direction \(T_{\Delta_S(p)}(\Delta_S(C))\).  Our variety \(X\) will be realized as a blow-up of \(S \times S \times T\), where \(T\) is a surface of general type; taking the product with \(T\) makes it simpler to control automorphisms of blow-ups.

We begin with a lemma enabling us to show that a blow-up \(\Bl_V X\) has no automorphisms except those that lift from \(X\).  Say that a variety \(X\) is \emph{\(\P^r\)-averse} if every \(\Kbar\)-morphism \(h : \P^r_{\Kbar} \to X_{\Kbar}\) is constant.
Note that if \(X\) is \(\P^r\)-averse, it is also \(\P^s\)-averse for any \(s > r\).

\begin{lemma}
\label{noextraautos}
Suppose that \(X\) is a \(\P^{r-1}\)-averse variety of dimension \(n\), and \(V \subset X\) is a smooth, equidimensional subvariety of codimension \(r\).  Write \(\pi : \Bl_V X \to X\) for the blow-up, with exceptional divisor \(E\).  Then every automorphism of \(\Bl_V X\) descends to an automorphism of \(X\), and the induced map \(\Aut(\Bl_V X) \to \Aut(X)\) is an isomorphism onto \(\Stab(V)\).
\end{lemma}
\begin{proof}
 We first observe that any nonconstant morphism \(h : \P^{r-1}_{\Kbar} \to \Bl_V X_{\Kbar}\) must have image contained in a geometric fiber of \(\pi\vert_{E_{\Kbar}}\).  Indeed, \(\pi \circ h : \P^{r-1}_{\Kbar} \to X_{\Kbar}\) must be constant since \(X\) is \(\P^{r-1}\)-averse, and so the image of \(h\) is contained in a geometric fiber.

Suppose that \(\phi : \Bl_V X \to \Bl_V X\) is an automorphism, and let \(h : \P^{r-1}_{\Kbar} \to \Bl_V X_{\Kbar}\) be the inclusion of a geometric fiber of \(\pi\vert_{E_{\Kbar}}\).  Then \(\phi \circ h\) is an inclusion from \(\P^{r-1}_{\Kbar} \to \Bl_V X_{\Kbar}\), and so must be the inclusion of some fiber of \(\pi\vert_{E_{\Kbar}}\).  Thus \(\phi\) permutes the fibers of \(\pi\vert_{E_{\Kbar}}\), and so descends to an automorphism of \(X\) fixing \(\pi(E) = V\).
\end{proof}

\begin{lemma}
\label{nopns} \leavevmode
\begin{enumerate}
\item Suppose that \(X_1\) and \(X_2\) are \(\P^r\)-averse.  Then \(X_1 \times X_2\) is \(\P^r\)-averse.
\item Suppose that \(X\) is \(\P^r\)-averse and \(V \subset X\) is a smooth, geometrically connected subvariety of codimension \(s \leq r\).  Then \(\Bl_V X\) is \(\P^r\)-averse.
\end{enumerate}
\end{lemma}
\begin{proof}
For (1), suppose that \(h : \P^r_{\Kbar} \to X_{1,\Kbar} \times X_{2,\Kbar}\) is a morphism.  Then the projections \(p_1 \circ h :  \P^r_{\Kbar} \to X_{1,\Kbar}\) and \(p_2 \circ h : \P^r_{\Kbar} \to X_{2,\Kbar}\) must both be constant, so that \(h\) is constant. For (2), the map \(\pi \circ h\) must be constant, and so if \(h\) is nonconstant, its image is contained in a fiber of \(\pi\vert_{E_{\Kbar}}\).  These fibers are isomorphic to \(\P^{s-1}_{\Kbar}\), and since \(s-1 < r\), the map \(h\) must be constant.
\end{proof}

We require one more simple lemma before proceding to the construction.
\begin{lemma}
\label{veryample}
Suppose that \(X\) is a smooth projective variety with \(\Aut(X)\) discrete.  There exists a divisor \(G \subset X\) for which \(\Aut(X;G)\) is trivial.
\end{lemma}
\begin{proof}
Choose a very ample linear system \(\mathcal G \cong \P^N\) on \(X\).  By the Lieberman--Fujiki theorem, the subgroup \(\Aut(X;\mathcal G)\) of automorphisms fixing \(\mathcal G\) is of finite type, and hence finite since \(\Aut(X)\) is assumed discrete.  If \(\phi\) is any member of \(\Aut(X,\mathcal G)\) other than the identity, it can not act trivially on \(\mathcal G\).
Indeed, suppose that \(\phi\) fixes every element of \(\mathcal G\).  If \(x\) is any point of \(X\), then \(x = \bigcap_{G \ni x} G\), and so \(x\) is fixed by \(\phi\).  It follows that \(\phi\) is the identity map.  As a consequence, a general element of \(\mathcal G\) is not fixed by any automorphisms.
\end{proof}

Let \(T\) be a smooth, geometrically simply connected surface over \(K\) for which \(\Aut(T)\) is trivial, \(T\) is not geometrically uniruled, and there is at least one \(K\)-point \(t\) on \(T\); according to~\cite{poonen}, we can take \(T\) to be the hypersurface in \(\P^3\) defined by \(x_0^5 + x_0 x_1^4 + x_1 x_2^4 + x_2 x_3^4 + x_3^5\), which has the point \([0,1,0,0]\).
(Note that if we work over \(K = \C\) or any other uncountable field, then any very general hypersurface in \(\P^3\) of degree at least \(4\) suffices.)

Take \(X_0 = S \times S \times T\).  The variety \(X\) will be constructed by a sequence of four blow-ups of \(X_0\).  In each case, the blowup satisfies the hypotheses of Lemma~\ref{noextraautos}, so we may identify its automorphism group with a subgroup of \(\Aut(X_0)\). 

\begin{lemma}
\label{blowupsequence}
Let \(X_0 = S \times S \times T\).  
Fix a point \(s\) on \(S\) and a divisor \(G\) on \(S\) with \(\Aut(S;G)\) trivial, as in Lemma~\ref{veryample}.  Choose three distinct smooth, geometrically connected curves \(\Gamma_1\), \(\Gamma_2\), and \(\Gamma_3\) in \(T\), and a point \(t\) on \(\Gamma_3\) which does not lie on \(\Gamma_1\) or \(\Gamma_2\).
\begin{enumerate}
\item The variety \(X_0\) is \(\P^r\)-averse for any \(r \geq 2\).  The automorphisms of \(X_0\) are of the form \(\Aut(S \times S) \times \id_T\).
\item Let \(\pi_1 : X_1 \to X_0\) be the blow-up of \(X_0\) along \(s \times S \times \Gamma_1\).  The variety \(X_1\) is \(\P^r\)-averse for any \(r \geq 3\).  The automorphisms of \(X_1\) are all lifts of \(\Aut(S;s) \times \Aut(S) \times \id_T\).
\item Let \(\pi_2 : X_2 \to X_1\) be the blow-up along the strict transform of \(G \times p \times \Gamma_2\).  The variety \(X_2\) is \(\P^r\)-averse for \(r \geq 4\).  The automorphisms of \(X_2\) are given by \(\id_S \times \Aut(S;p) \times \id_T\).
\item Let \(\pi_3 : X_3 \to X_2\) be the blow-up along the strict transform of \(p \times p \times \Gamma_3\).  Then \(X_3\) is \(\P^r\)-averse for \(r \geq 5\), and the automorphisms of \(X_3\) are of the form \(\id_S \times \Aut(S;p) \times \id_T\).
\item Let \(E_3\) be the exceptional divisor of \(\pi_3 : X_3 \to X_2\).  Then \(\Delta_S(C) \times t\) meets \(E_3\) at a single point \(u\).  Let \(\pi_4 : X_4 \to X_3\) be the blow-up at \(u\).    The automorphism group of \(X_4\) is isomorphic to \(\id_S \times \parauto \times \id_T\).
\end{enumerate}
\end{lemma}

\begin{proof}
We treat the blow-ups in order.
\begin{enumerate}[wide]
\item To show that \(X_0\) is \(\P^r\)-averse for \(r \geq 2\), it suffices to check that \(S\) and \(T\) are both \(\P^2\)-averse, according to the first part of Lemma~\ref{nopns}.  For \(T\) this follows since \(T\) is not uniruled, while for \(S\) we note that a nonconstant morphism \(h : \P^2_{\Kbar} \to S_{\Kbar}\) must be generically finite, and so induce an injection \(h^\ast : \Pic(S_{\Kbar}) \to \Pic(\P^2_{\Kbar})\), which is impossible.

Suppose that \(\phi : X_0 \to X_0\) is an automorphism. Let \(p_3 : X_0 \to T\) be the third projection.  We first claim that \(\phi\) must permute the geometric fibers of \(p_3\).  If \(p_3 \circ \phi\) contracts any geometric fiber of \(p_3\), it must contract every geometric fiber by the rigidity lemma.  So if \(\phi\) does not permute the fibers of \(p_3\), then every fiber of \(p_3\) has image in \(T\) of dimension at least \(1\).  Since these fibers are isomorphic to \(S \times S\), the image of every geometric fiber is uniruled, which implies that \(T\) must be geometrically uniruled, contradicting the choice of \(T\).

Consequently every automorphism of \(X_0\) is of the form \(\phi \times \psi\), where \(\phi\) is an automorphism of \(S \times S\) and \(\psi\) is an automorphism of \(T\).  Since \(\Aut(T)\) is trivial, the group \(\Aut(X_0)\) can be identified with \(\Aut(S \times S) \times \id_T\).

\item The center of the blow-up \(\pi_1\) has codimension \(3\), so it follows from Lemma~\ref{nopns} that \(X_1\) is \(\P^r\)-averse for \(r \geq 3\).  According to Lemma~\ref{noextraautos}, since \(X_0\) is \(\P^2\)-averse, \(\Aut(X_1)\) is given by the stabilizer of \(s \times S \times \Gamma_1\) in \(\Aut(X_0)\), which is isomorphic to the stabilizer of \(s \times S\) in \(\Aut(S \times S)\).

We claim that an element \(\phi\) of \(\Aut(S \times S)\) fixes \(s \times S\) only if it is of the form \(\phi_1 \times \phi_2\), where \(\phi_1\) is in \(\Aut(S;s)\) and \(\phi_2\) is in \(\Aut(S)\).  Indeed, if \(\phi\) fixes one fiber of \(p_1 : S \times S \to S\), it must permute the fibers, and so induces an automorphism \(\phi_1 : S \to S\) on the base with \(p_1 \circ \phi = \phi_1 \circ p_1\).  Then \((\id_S \times \phi_1^{-1}) \circ \phi\) is an automorphism of \(S \times S\) defined over \(p_1\).  This must be given by a map \(\id_S \times \phi_2 : S \times S \to S \times S\), since \(\Aut(S)\) is discrete,  and so \(\phi\) is of the form \(\phi_1 \times \phi_2\), where \(\phi_1\) fixes \(s\).

\item Since \(X_1\) is \(\P^r\)-averse for \(r \geq 3\) and \(X_2\) is the blow-up of \(X_1\) at a center of codimension \(4\), it follows that \(X_2\) is \(\P^r\)-averse for \(r \geq 4\).  Since the center of \(\pi_2\) has codimension \(4\) and \(X_1\) is \(\P^3\)-averse, the automorphisms of \(X_2\) are given by isomorphisms of \(X_1\) that fix \(G \times p \times t_2\).  The automorphisms of \(X_1\) are all of the form \(\phi_1 \times \phi_2 \times \id_T\), and so this stabilizer is exactly \(\id_S \times \Aut(S;p) \times \id_T\). 

\item We have seen that \(X_2\) is \(\P^4\)-averse, and \(X_3\) is the blow-up of \(X_2\) at a center of codimension \(5\).  It follows that \(X_3\) is \(\P^r\)-averse for \(r \geq 5\), and the automorphisms of \(X_3\) are lifts of automorphisms of \(X_2\) that fix \(p \times p \times \Gamma_3\).  Every automorphism of \(X_2\) fixes \(p \times p \times \Gamma_3\), and so the automorphisms of \(X_3\) are again given by \(\id_S \times \Aut(S;p) \times \id_T\). 

\item The centers of the blow-ups \(\pi_1\) and \(\pi_2\) are both disjoint from the fiber \(S \times S \times t\), since \(t\) lies on neither \(\Gamma_1\) nor \(\Gamma_2\), while the center of the blow-up \(\pi_3\) meets \(S \times S \times t\) at the single point \(p \times p \times t\).  As a result, \(\Delta_S(C) \times t\) meets \(E_3\) at one point \(u\), as claimed.  The restriction of \(\pi_3 \circ \pi_2 \circ \pi_1\) to the strict transform of \(S \times S \times t\) is the blow-up at the point \(p \times p \times t\).

Since \(X_3\) is \(\P^5\)-averse and the center of \(\pi_3\) has codimension \(6\), \(\Aut(X_4)\) is isomorphic to the stabilizer of \(u\) in \(\Aut(X_3)\).  These are exactly the automorphisms \(\id_S \times \phi \times \id_T\) of \(X_3\) that fix the tangent direction \(T_{\Delta(p)}(\Delta_S(C)) \times t\).  According to Lemma~\ref{fixingp}, these are exactly the lifts of automorphisms of the form \(\id_S \times \parauto \times \id_T\). \qedhere
\end{enumerate}
\end{proof}

This completes the construction.

\begin{proof}[Proof of Theorem~\ref{autoexample}]
Let \(X = X_4\) be as in Lemma~\ref{blowupsequence}.  The variety \(X\) is smooth, projective and geometrically simply connected, since it is a blow-up of \(S \times S \times T\) where \(S\) is a rational surface and \(T\) is smooth and geometrically simply conneected.  The group \(\Aut(X)\) is isomorphic to \(\parauto\), which is not finitely generated according to Lemma~\ref{surfacenonfg}.
\end{proof}

\section{A variety with many forms}

We now show how the construction of the previous section can be adapted to give an example of a \(K\)-variety with infinitely many \(L/K\)-forms even when \(L/K\) is a finite extension.  In the case \(K = \R\) and \(L = \C\), we obtain an example of a variety with infinitely many non-isomorphic real structures.

A standard descent argument shows that the \(L/K\)-forms of \(X\) are classified by the Galois cohomology  \(H^1(\Gal(L/K),\Aut(X_L))\)~\cite{serregalois}.  In many settings, this set is finite.  Indeed, according to a theorem of Borel and Serre~\cite[Th\'eor\`eme 6.1]{borelserre}, if \(\pi_0(\Aut(X_{\bar{K}}))\) is an arithmetic group, then the set of forms of \(X\) over \(\bar{K}\) is finite; this includes nearly all varieties for which the group of automorphisms is known.  The set of forms is also finite when \(X\) is a minimal surface of non-negative Kodaira dimension~\cite{dik}, even though  for such varieties the group of automorphisms need not even be commensurable with an arithmetic group~\cite{totaro}.

Our example of a variety with infinitely many forms is obtained by an additional blow-up of the variety \(X\) constructed in the first section.  When \(L/K\) is a quadratic extension and every automorphism of \(X_L\) is defined over \(K\), the set \(H^1(\Gal(L/K),\Aut(X_L))\) is simply the set of conjugacy classes of involutions in \(\Aut(X_L)\). The next lemma makes explicit what is required.

\begin{lemma}
\label{neededproperties}
Suppose that \(L/K\) is a quadratic extension, and that \(X\) is a smooth, projective variety over \(K\).  Suppose that there is a finite-index subgroup \(G^\prime \subset \Aut(X_L)\) which contains infinitely many conjugacy classes of involutions and on which \(\Gal(L/K)\) acts trivially.  Then the variety \(X\) has infinitely many \(L/K\)-forms.
\end{lemma}
\begin{proof}
The forms of \(X\) are classified by the set \(H^1(\Gal(L/K),\Aut(X_L))\). Because the action of \(\Gal(L/K)\) on \(G^\prime\) is trivial, \(H^1(\Gal(L/K),G^\prime)\) is the set of conjugacy classes of involutions in \(G^\prime\), which is infinite by assumption.  There is an exact sequence
\[
H^0(\Gal(L/K),\Aut(X_L)/G^\prime) \to H^1(\Gal(L/K),G^\prime) \to H^1(\Gal(L/K),\Aut(X_L))
\]
Here \(\Aut(X_L)/G^\prime)\) should be interpreted as the set of left-conjugacy classes of \(G^\prime\) rather than a group, but the sequence is nevertheless exact~\cite[Proposition 38]{serregalois}.  Since \(G^\prime\) has finite index in \(\Aut(X_L)\), the leftmost set is finite, whence \(H^1(\Gal(L/K),\Aut(X_L))\) is infinite, as claimed.
\end{proof}

We retain the notation of the first section, labelling the six lines as \(L_0,\ldots,L_5\), with \(L_0\) the curve \(C\) of Theorem~\ref{surfaceautos}.  Let \(p_{ij} = L_i \cap L_j\), and write \(p_i\) for the point \(p_{0i}\).  Recall that the other lines \(L_i\) are chosen so that the intersections of \(L_1\), \(L_2\), \(L_3\), \(L_4\) and \(L_5\) with \(L_0\) are given by the points \(p_1 = 0\), \(p_2 = 1\), \(p_3 = 2\), \(p_4 = 3\), and \(p_5 = 6\), in suitable coordinates on \(L_0\).  We will consider the following subgroups of \(\Aut(S)\):

\begin{enumerate}
\item \(\parauto\), the subgroup of automorphisms restricting to \(L_0\) as \(z \mapsto z+c\);
\item \(\bigauto\), the subgroup of automorphisms restricting to \(L_0\) as either \(z \mapsto z+c\) or \(z \mapsto -z+c\);
\item \(\bigautoev\), the subgroup of automorphisms which fix the two lines \(L_0\) and \(L_5\) as well as the curves \(L_1 \cup L_4\) and \(L_2 \cup L_3\).
\end{enumerate}

Recall that every automorphism of \(S\) must permute the six lines \(L_i\) since their union is the unique member of \(\abs{-2K_S}\); an automorphism lies in \(\bigautoev\) if it fixes \(L_0\) and \(L_5\) and either fixes or exchanges the members of the two other pairs.  In particular, \(\bigautoev\) has finite index in \(\bigauto\).

Let \(s_0 : S \to S\) be the involution of \(S\) determined by the marking
given by the triples \(L_0,L_1,L_4\) and \(L_5,L_2,L_3\), with \(L_0\) and \(L_5\) as the distinguished elements.  The automorphism \(s_0\) fixes the two distinguished lines \(L_0\) and \(L_5\), and exchanges \(L_1\) with \(L_4\) and \(L_2\) with \(L_3\).
This map restricts to \(L_0\) in such a way that it exchanges \(p_1 = 0\) with \(p_4 = 3\) and \(p_2 = 1\) with \(p_3 = 2\); thus the restriction is the involution \(z \mapsto 3-z\), and \(s_0\) is contained in the subgroup \(\bigauto\).
Figure~\ref{s0picture} shows the important curves in \(\P^2\) acted on by the map \(s_0\).  The two dashed lines are exchanged, as are the two heavily dotted lines.  The pencil of lines passing through the point \(p_5\) is preserved by \(s_0\).  The strict transforms of the two lightly dotted lines through \(p_5\) are \((-1)\)-curves on \(S\) with classes \(H - E_{05} - E_{14}\) and \(H- E_{05} - E_{23}\), which will appear later.

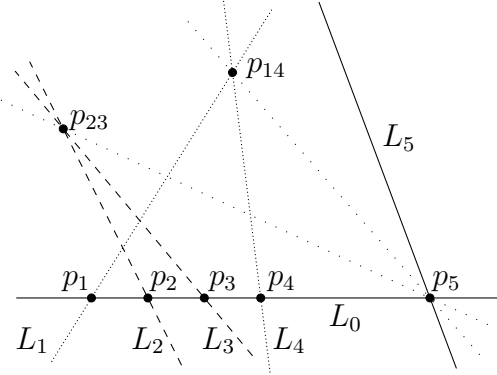
\begin{figure}[hbt]
\centering
\begin{tikzpicture}[scale=0.75, extended line/.style={shorten >=-#1,shorten <=-#1}, extended line/.default=1cm]

\coordinate (P1) at (0,0);
\filldraw (P1) circle (2pt) node[shift={(-0.2,0.22)}] {$p_1$};
\coordinate (P2) at (1,0);
\filldraw (P2) circle (2pt) node[shift={(0.22,0.22)}] {$p_2$};
\coordinate (P3) at (2,0);
\filldraw (P3) circle (2pt) node[shift={(0.25,0.22)}] {$p_3$};
\coordinate (P4) at (3,0);
\filldraw (P4) circle (2pt) node[shift={(0.28,0.22)}] {$p_4$};
\coordinate (P5) at (6,0);
\filldraw (P5) circle (2pt) node[shift={(0.22,0.22)}] {$p_5$};

\draw [extended line] (P1)--(P5);

\coordinate (X1) at (2.5,4);
\filldraw (X1) circle (2pt) node[shift={(0.45,0.05)}] {$p_{14}$};
\draw [extended line, densely dotted] (P1) -- (X1);
\draw [extended line, densely dotted] (P4) -- (X1);
\coordinate (X2) at (-0.5,3);
\filldraw (X2) circle (2pt) node[shift={(0.35,0.11)}] {$p_{23}$};
\draw [extended line, dashed] (P2) -- (X2);
\draw [extended line, dashed] (P3) -- (X2);

\draw [extended line, loosely dotted] (P5) -- (X1);
\draw [extended line, loosely dotted] (P5) -- (X2);

\draw [extended line] (P5) -- (4.5,4);

\node at (-1.05,-0.75) {$L_1$};
\node at (1.00,-0.75) {$L_2$};
\node at (2.25,-0.75) {$L_3$};
\node at (3.50,-0.75) {$L_4$};
\node at (4.5,-0.35) {$L_0$};
\node at (5.45,2.84) {$L_5$};
\end{tikzpicture}
\caption{The involution $s_0$}
\label{s0picture}
\end{figure}

Let \(\tilde{s}_0 : X_3 \to X_3\) be the automorphism of \(X_3\) induced by \(\id_S \times s_0 \times \id_T\), under the identification of Lemma~\ref{blowupsequence}.  The variety \(X\) of the first section was obtained by blowing up a point \(u\) on \(E_3\).  It will now be convenient to blow up a second such point \(\tilde{s}_0(u)\) as well.

\begin{lemma}.
Let \(X^\prime\) be the blowup of \(X_3\) at \(u\) and \(\tilde{s}_0(u)\).  Then
 \(\Aut(X^\prime) \cong \bigauto\).
\end{lemma}
\begin{proof}
According to Lemma~\ref{noextraautos}, the automorphisms of \(X^\prime\) are the stabilizer of \(u \cup \tilde{s}_0(u)\).  These are precisely the automorphisms \(S\) which are of either the form \(z \mapsto z+c\) or \(z \mapsto -z+c\), as required.
\end{proof}

Let \(\bar{\tau} : S \to S\) be an automorphism restricting to \(L_0\) as \(\tau = (z \mapsto 3z)\).  The elements \(s_n = \bar{\tau}^{-n} \circ s_0 \circ \bar{\tau}^{n}\) are all involutions, and the restriction of \(s_n\) to \(L_0\) is given by the map \(z \mapsto 3^{1-n}- z\), which lies in \(\bigauto\).  Although the maps \(s_n\) are conjugate in \(\Aut(S)\), they are conjugate by powers of \(\bar{\tau}\), and \(\bar{\tau}\) is not contained in \(\bigauto\).  We now work to show that the \(s_n\) indeed define distinct conjugacy classes in the subgroup \(\bigauto\).

Since it is difficult to study relations in \(\Aut(S)\) directly (cf.\ Remarks~\ref{notinjective} and~\ref{badsurface}), it will be convenient to consider the action of this group on \(N^1(S)\).   A basis for \(N^1(S)\) given by \(H = \pi^\ast \mathcal O_{\P^2}(1)\) followed by the fifteen classes \(E_{ij}\), which we order lexicographically.

\begin{lemma}
\label{uniquenef}
The class \(H-E_{05}\) is the unique class \(D\) in \(N^1(S)\) for which:
\begin{enumerate}
\item \(D\) is contained in the \((+1)\)-eigenspace of the involution \(s_0^\ast : N^1(S) \to N^1(S)\).
\item \(D \cdot L_0 = D \cdot L_5 = 0\) and \(D \cdot L_1 = D \cdot L_2 = D \cdot L_3 = D \cdot L_4 = 1\).
\item \(D\) is nef.
\item \(D^2 = 0\).
\end{enumerate}
\end{lemma}

\begin{proof}
The linear system \(H- E_{05}\) is given by the strict transforms on \(S\) of the pencil of lines through \(p_{05}\).  Since this pencil is preserved by \(s_0\) and the linear system on \(S\) is basepoint-free, the claimed properties follow.  We next check that there are no other classes with this property.

Suppose that \(D\) is a class satisfying all these conditions.  Condition (1) limits \(D\) to an \(8\)-dimensional subspace of the \(16\)-dimensional vector space \(N^1(S)\), and condition (2) further restricts the class of \(D\) to a \(4\)-dimensional affine subspace of \(N^1(S)\).  By condition (3), \(D\) must have positive intersection with the fifteen classes \(E_{ij}\), and the two \((-1)\)-curve classes \(H - E_{05} - E_{14}\) and \(H- E_{05} - E_{23}\).
The  class \(D\) is thus constrained to lie in the intersection of \(17\) closed half-spaces inside a \(4\)-dimensional affine space.  A routine calculation shows that the intersection of these half-spaces is a four-dimensional, unbounded polyhedron in \(N^1(S)\) with a single vertex \(H-E_{05}\) and six infinite rays \(R_i\).  The six rays are given by the rows of the matrix
\[
\left(
\begin{array}{rrrrrrrrrrrrrrrrr}
5 & -1 & 0 & 0 & -1 & -3 & -1 & -1 & -2 & 0 & -2 & -1 & -1 & -1 & -1 & 0 \\
5 & 0 & -1 & -1 & 0 & -3 & -1 & -1 & -2 & -1 & -2 & -1 & 0 & -1 & 0 & -1 \\
4 & -1 & 0 & 0 & -1 & -2 & -1 & -1 & 0 & -1 & -2 & -1 & 0 & -1 & 0 & -1 \\
4 & 0 & -1 & -1 & 0 & -2 & -1 & -1 & -2 & 0 & 0 & -1 & -1 & -1 & -1 & 0 \\
3 & -1 & 0 & 0 & -1 & -1 & -1 & -1 & 0 & 0 & 0 & -1 & -1 & -1 & -1 & 0 \\
3 & 0 & -1 & -1 & 0 & -1 & -1 & -1 & 0 & -1 & 0 & -1 & 0 & -1 & 0 & -1 
\end{array}
\right)
\]
Any class satisfying conditions (1)--(3) is the sum of \(H-E_{05}\) and some non-negative linear combination of the six rays \(R_i\).  However, we find that \((H-E_{05})^2 = 0\), \((H-E_{05}) \cdot R_i = 2\), and \(R_i \cdot R_j \geq 0\) for any \(i\) and \(j\), not necessarily distinct.  It follows that the only class in this polyhedron with self-intersection \(0\) is the vertex itself.
\end{proof}

An accompanying Sage file provides a computation of the matrix for the action of \(s_0^\ast\) on \(N^1(S)\) with respect to our basis, as well as the calculation yielding the polyhedron and the intersections of the generating rays \(R_i\).

\begin{lemma}
The centralizer of \(s_0\) in \(\bigautoev\) is given by \(\langle \id_S, s_0 \rangle\).
\end{lemma}

\begin{proof}

Suppose that \(\phi : S \to S\) is an automorphism of \(S\) commuting with \(s_0\).  Then \(\phi^\ast : N^1(S) \to N^1(S)\) must preserve the \((+1)\)-eigenspace of \(s_0^\ast\), and so \(\phi^\ast(H-E_{05})\) lies in this eigenspace as well.  If \(\phi\) lies in \(\bigautoev\), then the intersection property (2) must be satisfied by \(\phi^\ast(H-E_{05})\).  Since \(\phi^\ast\) also preserves the nef cone and the intersection form, it in fact satisfies the hypotheses (1)--(4) of Lemma~\ref{uniquenef}.  It then follows from the lemma that \(\phi^\ast(H-E_{05}) = H-E_{05}\) in \(N^1(S)\), and since \(\Pic^0(S)\) is trivial, that \(\phi\) must preserve the class \(H-E_{05}\) in \(\Pic(S)\).
As a result, \(\phi\) permutes the fibers of the map \(\lambda : S \to \P^1\) given by the basepoint-free linear system \(\abs{H-E_{05}}\).

In particular, \(\phi\) permutes the singular fibers of \(\lambda\).  It must map \(E_{34}\) to another \(s_0\)-invariant \((-1)\)-curve contained in a fiber of \(\phi\) that has intersection \(1\) with both \(L_1\) and \(L_4\), and \(0\) with \(L_2\) and \(L_3\).  The only two such curves are \(E_{14}\) itself and the strict transform of the line from \(p_{05}\) to \(p_{23}\), which has class \(H-E_{05} - E_{23}\).  However, under the generality hypothesis of Remark~\ref{extragenerality}, there is no map that fixes \(L_0\) and \(L_5\) while exchanging the fibers containing these curves; consequently each of these fibers must be mapped to itself.  This implies that \(\phi\) fixes four fibers of the map \(\lambda\), and since the base is \(\P^1\), that \(\phi\) maps every fiber of \(\lambda\) to itself.

Since by assumption \(\phi\) lies in the subgroup \(\bigautoev\), it fixes the two curves \(L_0\) and \(L_5\) and either fixes or exchanges the members of the pairs \(L_1\), \(L_4\) and \(L_2\), \(L_3\).  Replacing \(\phi\) with \(\phi \circ s_0\) if necessary, we obtain an element commuting with \(s_0\) which fixes the four curves \(L_0\), \(L_5\), \(L_1\), and \(L_4\), and either fixes the two curves \(L_2\) and \(L_3\) or exchanges them.

Now, there is a singular fiber of \(\lambda\) containing the two \((-1)\)-curves with classes \(E_{12}\) and \(H - E_{05} - E_{12}\).  Suppose that \(\phi\) exchanges the two sections \(L_2\) and \(L_3\).  
Since
\begin{align*}
L_1 \cdot E_{12} &= 1,& L_2 \cdot E_{12} &= 1,& L_3 \cdot E_{12} &= 0,& L_4 \cdot E_{12} &= 0,
\end{align*}
it must be that
\begin{align*}
\phi(L_1) \cdot \phi(E_{12}) &= 1,& \phi(L_2) \cdot \phi(E_{12}) &= 1,& \phi(L_3) \cdot \phi(E_{12}) &= 0,& \phi(L_4) \cdot \phi(E_{12}) &= 0, 
\end{align*}
and so
\begin{align*}
L_1 \cdot \phi(E_{12}) &= 1,& L_3 \cdot \phi(E_{12}) &= 1,& L_2 \cdot \phi(E_{12}) &= 0,& L_4 \cdot \phi(E_{12}) &= 0.
\end{align*}
The fibers of \(\lambda\) are preserved, and so \(\phi(E_{12})\) must be either \(E_{12}\) or the curve of class \(H - E_{05} - E_{12}\).  However, neither of these curves has the requisite intersection properties for \(\phi(E_{12})\).  We conclude that \(\phi\) must send \(L_2\) to itself and \(L_3\) to itself.

Thus \(\phi\) must commute with the projection \(\lambda\) and fix the four sections \(L_1\), \(L_2\), \(L_3\), and \(L_4\).  A general geometric fiber \(F\) of \(\lambda\) is a rational curve in the linear system \(\abs{H - E_{05}}\).  The map \(\phi\) fixes the four points of intersection of \(F\) with the sections listed, and so \(\phi\vert_F\) must be the identity map.  Since \(\phi\) fixes a Zariski dense set of points on \(S_L\), it must be the identity.  As we may have previously replaced \(\phi\) with \(\phi \circ s_0\), we conclude that the centralizer is given by \(\langle \id_S,s_0 \rangle\).
\end{proof}

\begin{corollary}
\label{manyinvolutions}
The group \(\bigautoev\) contains infinitely many conjugacy classes of involutions.
\end{corollary}

\begin{proof}
Let \(\bar{\tau}\) be an automorphism of \(S\) restricting to \(z \mapsto 3z\) on \(L_0\). Since any automorphism of \(S\) permutes the six lines, there exists some \(N > 0\) for which the iterate \(\bar{\tau}^N\) maps each of the six lines \(L_i\) to itself.  The map \(s_{Nn} =  \bar{\tau}^{-Nn} \circ s_0 \circ \bar{\tau}^{Nn}\) is an involution which restricts to \(L_0\) as \(z \mapsto 3^{1-Nn} - z\), and since \(s_{Nn}\) induces the same permutation of the \(L_i\) as does \(s_0\), it lies in the subgroup \(\bigautoev\).

We claim that no two distinct \(s_{Nm}\) and \(s_{Nn}\) are conjugate by an element of \(\bigautoev\).  It suffices to show that \(s_0\) is not conjugate to any \(s_{Nn}\).  If \(s_{Nn} = \bar{\tau}^{-Nn} \circ s_0 \circ \bar{\tau}^{Nn} = \alpha \circ s_0 \circ \alpha^{-1}\) for some \(\alpha\), then \(\bar{\tau}^{Nn} \circ \alpha\) commutes with \(s_0\).  According to the previous lemma, either \(\alpha = \bar{\tau}^{-Nn}\) or \(\alpha = \bar{\tau}^{-Nn} \circ s_0\).  Since neither \(\bar{\tau}^{-Nn}\) nor \(\bar{\tau}^{-Nn} \circ s_0\) is contained in \(\bigautoev\) for any nonzero value of \(n\), the claim follows.
\end{proof}

\begin{lemma}
\label{autosreal}
Every automorphism of \(X^\prime_L\) is defined over \(K\).
\end{lemma}

\begin{proof}
Since \(S_L\) is constructed by blowing up \(K\)-points in \(\P^2\), its Picard group is generated by the classes of \(K\)-divisors.  The Galois action on \(\Pic(S_L)\) is thus trivial, and preserves the class of any \((-1)\)-curve.  Because each \((-1)\)-curve is rigid in its cohomology class, these curves are invariant under the conjugation map \(c : S_L \to S_L\).

Suppose that \(\phi : X^\prime_L \to X^\prime_L\)  is any automorphism.  Then \(\phi\) is induced by some automorphism \(\psi : S_L \to S_L\), and  \(c \circ \psi \circ c : S_L \to S_L\) is an \(K\)-automorphism which fixes each \((-1)\)-curve \(E\).  Since \(\Pic(S)\) is generated by classes of real \((-1)\)-curves, it follows that \(\psi\) and \(c \circ \psi \circ c\) have the same action on \(\Pic(S)\).  As \(H^0(S_L,TS_L) = 0\), these two maps must coincide, so that \(c \circ \psi = \psi \circ c\), and \(\psi\) is defined over \(K\).
\end{proof}

\begin{proof}[Proof of Theorem~\ref{twistexample}]
We have \(\Aut(X^\prime_L) \cong \bigauto\), and \(\bigautoev\) is a finite index subgroup of \(\bigauto\) on which \(\Gal(L/K)\) acts trivially. By Lemma~\ref{manyinvolutions}, \(\bigautoev\) contains infinitely many conjugacy classes of involutions, and Theorem~\ref{twistexample} then follows from Lemma~\ref{neededproperties}.
\end{proof}

\section{Acknowledgements}

I am deeply grateful to Igor Dolgachev for a number of discussions related to this problem.  Several people made useful comments on an earlier version, including Serge Cantat, Jeff Diller, Daniel Litt, James M\textsuperscript{c}Kernan, Roberto Svaldi, and Burt Totaro.   I have also benefited from discussions with Piotr Pokora on these rational surfaces. The participants in the MathOverflow thread~\cite{mothread} provided some useful context.

\singlespacing
\bibliographystyle{amsplain}
\bibliography{refs}

\end{document}